\pdfoutput=1
\RequirePackage{ifpdf}
\ifpdf 
\documentclass[pdftex]{sigma}
\else
\documentclass{sigma}
\fi

\usepackage[all]{xy}

\numberwithin{equation}{section}

\newtheorem{Theorem}{Theorem}[section]
\newtheorem{Proposition}[Theorem]{Proposition}
\newtheorem{Lemma}[Theorem]{Lemma}
\newtheorem{Corollary}[Theorem]{Corollary}
{\theoremstyle{definition}
\newtheorem{Remark}[Theorem]{Remark}}

\begin{document}

\allowdisplaybreaks

\newcommand{\arXivNumber}{1508.02318}

\renewcommand{\PaperNumber}{072}

\FirstPageHeading

\ShortArticleName{Cohomology of the Moduli Space of Rank Two}

\ArticleName{Cohomology of the Moduli Space of Rank Two,\\ Odd Degree Vector Bundles over a Real Curve}

\Author{Thomas John BAIRD}

\AuthorNameForHeading{T.J.~Baird}

\Address{Department of Mathematics and Statistics, Memorial University of Newfoundland,\\
St.~John's, NL, A1C 5S7, Canada}
\Email{\href{mailto:tbaird@mun.ca}{tbaird@mun.ca}}
\URLaddress{\url{https://thomasjohnbaird.com}}

\ArticleDates{Received October 30, 2015, in f\/inal form July 20, 2016; Published online July 22, 2016}

\Abstract{We consider the moduli space of rank two, odd degree, semi-stable Real vector bundles over a real curve, calculating the singular cohomology ring in odd and zero characteristic for most examples.}


\Classification{53D30; 55R10; 55T20}

\section{Introduction} A \emph{real curve} $(\Sigma, \tau)$ is a closed Riemann surface $\Sigma$ of genus $g$, endowed with an anti-holomorphic involution $\tau$. The \emph{real points} of $(\Sigma, \tau)$ are those f\/ixed by $\tau$,
\begin{gather*}\Sigma^{\tau} = \{ p \in \Sigma\,|\, \tau(p)= p\}.\end{gather*}

 A \emph{Real $C^{\infty}$-vector bundle} $(E,\tilde{\tau})$ over $(\Sigma,\tau)$ is a complex $C^{\infty}$-vector bundle $E \rightarrow \Sigma$ equipped along with a smooth lift
\begin{gather*}\xymatrix{\ar[r]^{\tilde{\tau}} E \ar[d] & E \ar[d] \\
 \Sigma \ar[r]^{\tau} & \Sigma, }\end{gather*}
which is conjugate linear on f\/ibres and such that $\tilde{\tau} \circ \tilde{\tau} = \operatorname{Id}_E$. A \emph{Real holomorphic vector bundle} is a Real $C^{\infty}$-vector bundle equipped with a holomorphic structure for which $\tilde{\tau}$ is anti-holomorphic. \emph{Quaternionic}~$C^{\infty}$/\emph{holomorphic vector bundles} are def\/ined similarly, except that one requires $\tilde{\tau} \circ \tilde{\tau} = -\operatorname{Id}_E$, rather than $\operatorname{Id}_E$.

Fix a complex $C^{\infty}$-vector bundle $E \rightarrow \Sigma$ with rank $r$ and degree $d$ such that $\operatorname{gcd}(r,d) =1$. The moduli space of semi-stable complex vector bundles
\begin{gather*} M(E) = M(r,d)\end{gather*}
is a compact, connected K\"ahler manifold (indeed a projective variety), that parametrizes isomorphism classes of semi-stable holomorphic vector bundles of $C^{\infty}$-type $E$.

{\sloppy Similarly, f\/ix a Real/Quaternionic $C^{\infty}$-vector bundle $(E,\tilde{\tau})$ with coprime rank and degree. The \emph{moduli space of semi-stable Real}/\emph{Quaternionic vector bundles} \begin{gather*}M(E,\tilde{\tau}) = M(r,d,\tilde{\tau})\end{gather*}
is a compact, connected manifold that parametrizes isomorphism classes of semi-stable Real/Qua\-ter\-nio\-nic holomorphic vector bundles with $C^{\infty}$-isomorphism type $(E,\tilde{\tau})$. The forgetful map
\begin{gather*}\iota_{\tilde{\tau}}\colon \ M(r,d,\tilde{\tau}) \hookrightarrow M(r,d),\end{gather*}
embeds $M(r,d,\tilde{\tau})$ as a real, Lagrangian submanifold of $M(r,d)$.

}

Given a holomorphic vector bundle ${\mathcal E}$ over $\Sigma$, the \emph{conjugate pull-back} $\tau^*\overline{{\mathcal E}}$ is a holomorphic vector bundle over $\Sigma$ of the same degree. This induces an anti-holomorphic, anti-symplectic involution on $M(r,d)$ (which we also denote $\tau$) sending the equivalence class $[{\mathcal E}]$ to \begin{gather*}\tau([{\mathcal E}]) = [\tau^*\overline{{\mathcal E}}].\end{gather*} The f\/ixed point set $M(r,d)^{\tau}$ is a disjoint union of real, Lagrangian submanifolds. It was proven by Biswas--Huisman--Hurtubise \cite{BHH} and independently by Schaf\/fhauser~\cite{S11, S12}, that the embeddings $\iota_{\tilde{\tau}}$ produce a dif\/feomorphism
\begin{gather*} \coprod_{\tilde{\tau}} M(r,d,\tilde{\tau}) \cong M(r,d)^{\tau},\end{gather*}
where the coproduct is indexed by $C^{\infty}$-isomorphism types $\tilde{\tau}$ lifting $\tau$. Hereafter, we abuse notation and identify $M(r,d,\tilde{\tau})= M(E,\tilde{\tau})$ with its embedded image.

The ${\mathbb Z}_2$-Betti numbers of $M(r,d,\tilde{\tau})$ were calculated by Liu--Schaf\/fhauser \cite{LS} and independently by the author \cite{B} using the real Harder--Narasimhan stratif\/ication (described in Section~\ref{realstratsect}). In the current paper, we use this stratif\/ication to compute the cohomology ring $H^*(M(2,d,\tilde{\tau}); k)$ for rank two bundles of odd degree $d$, and coef\/f\/icient f\/ields $k$ of characteristic $\neq 2$, for most isomorphism types~$\tilde{\tau}$ (see Theorem~\ref{bigthm})\footnote{The condition that $d$ is odd forces $(\Sigma, \tau)$ to have real points and excludes the possibility that the bundle is quaternionic (see Remark~\ref{quatnotint}), facts we exploit in the proof of our main result Theorem~\ref{bigthm}. However, many of our key results (f\/irst part of Theorem~\ref{isomcoh}, Proposition~\ref{orientableProp} and Theorem \ref{bgring}) are independent of $d$ and provide information about the moduli stack underlying $M(2,d,\tilde{\tau})$ when $d$ is even. See for example Theorem~\ref{QuatDone}, on the moduli stack of Quatenionic bundles.}.

The approach follows the general lines of approach laid out by Atiyah and Bott \cite{AB} in the complex case. We have the space ${\mathcal C}^{\tilde{\tau}} = {\mathcal C}(E,\tilde{\tau})$ of Cauchy--Riemann operators on $E$ that commute with $\tilde{\tau}$. This is acted on by the group ${\mathcal G}^{\tilde{\tau}} = {\mathcal G}(E,\tilde{\tau})$ of gauge transformations that commute with $\tilde{\tau}$, and one is interested in the quotients under the action of this group. This leads us f\/irst to consider the ${\mathcal G}^{\tilde{\tau}}$-equivariant cohomology $H^*_{{\mathcal G}^{\tilde{\tau}}}$. For the total space of all operators, which is contractible, the equivariant cohomology is simply the singular cohomology of the classifying space $B{\mathcal G}^{\tilde{\tau}}$; one then must consider a stratif\/ication, and ``remove" certain unstable strata, to obtain the equivariant cohomology of the semistable stratum ${\mathcal C}_{ss}^{\tilde{\tau}} \subseteq {\mathcal C}^{\tilde{\tau}}$. A f\/inal step is to relate this equivariant cohomology to the ordinary cohomology of the moduli space $M(E,\tilde{\tau}) := {\mathcal C}_{ss}^{\tilde{\tau}}/ {\mathcal G}^{\tilde{\tau}}$.

One issue is that the unstable strata no longer necessarily have orientable normal bundles, and so the cohomology of the Thom space of these strata, which is their contribution to the global cohomology, is no longer given by the Thom isomorphism. In the case we consider, that of rank two bundles, this turns out to be an advantage instead of a handicap, because it forces the Thom spaces to be acyclic. This yields the following result.

\begin{Theorem}\label{isomcoh}
	Let $(E,\tilde{\tau})$ be a Real/Quaternionic $C^{\infty}$-vector bundle over a real curve $(\Sigma,\tau)$ of rank two. If the normal bundles of the unstable strata of the real Harder--Narasimhan stratification are all nonorientable, then there is a natural cohomology isomorphism
\begin{gather*}
 H^*(B{\mathcal G}(E,\tilde{\tau}); k) \cong H_{{\mathcal G}(E,\tilde{\tau})}^*({\mathcal C}_{ss}(E,\tilde{\tau}); k)
\end{gather*}
for coefficient fields $k$ of characteristic $\neq 2$. If additionally $E$ has odd degree, then
\begin{gather*}
H^*(B{\mathcal G}(E,\tilde{\tau}); k) \cong H^*(M(E,\tilde{\tau}); k).
\end{gather*}
\end{Theorem}

We prove that the hypotheses of Theorem~\ref{isomcoh} hold except in certain special cases (see Proposition~\ref{orientableProp}). This reduces our problem to calculating $H^*(B{\mathcal G}(E,\tilde{\tau}); k)$, which is done using an Eilenberg--Moore spectral sequence in Section~\ref{SScalcReal}. For odd degree bundles, this determines the cohomology ring for most isomorphism types.

\begin{Theorem}\label{bigthm}
Let $(\Sigma,\tau)$ be a real curve of genus $g\geq 2$, let $d$ be odd, and let $k$ be a field of odd or zero characteristic. Then the cohomology ring $ H^*(M(2,d,\tilde{\tau}); k)$ is an exterior algebra with~$g$ generators of degree~$1$ and $(g-1)$ of degree~$3$ for all but one exceptional $C^{\infty}$-type $\tilde{\tau}$ which occurs only when $(\Sigma,\tau)$ is a type~I curve of even genus.
If the genus $g \geq 4$, then for that one exceptional type, $H^*(M(2,d,\tilde{\tau}); k)$ is not an exterior algebra.
\end{Theorem}

The one exceptional $C^{\infty}$-type referred to in Theorem \ref{bigthm} is distinguished by the property that $E^{\tilde{\tau}}$ is non-orientable over every path component of $\Sigma^{\tau}$.

At f\/irst encounter, Theorem \ref{bigthm} is a little disappointing. For most $C^{\infty}$-types, the rational cohomology ring is simply an exterior algebra that does not depend on the real structure $\tau$. In contrast, the ${\mathbb Z}_2$-cohomology has interesting Betti numbers that do depend on $\tau$. On might be tempted to conclude that the odd and zero characteristic cohomology contains no interesting information.

However, we are led to a dif\/ferent conclusion if we consider moduli spaces of bundles with f\/ixed determinant. Given a Real vector bundle $(E,\tilde{\tau})$, the determinant line bundle $\det(E)$ inherits a real structure from~$\tilde{\tau}$. This gives rise to a natural f\/ibre bundle map
\begin{gather*} \det \colon \ M(r,d,\tilde{\tau}) \rightarrow M(1,d,\tilde{\tau}).\end{gather*}

The f\/ibre $M_{\Lambda}(r,d,\tilde{\tau}):= \det^{-1}(\Lambda)$ is called the moduli space of Real vector bundles of f\/ixed determinant $\Lambda$. The (f\/inite) group, $T_r$, of $r$-th roots of the trivial Real line bundle acts on $M_{\Lambda}(r,d,\tilde{\tau})$ by tensor product. The following corollary follows from (and is equivalent to) Theorem~\ref{bigthm}.

\begin{Corollary}\label{BigCor}
Under the hypotheses of Theorem {\rm \ref{bigthm}}, the ring of $T_2$-invariants \begin{gather*} H^*(M_{\Lambda}(2,d,\tilde{\tau}); k)^{T_2}\end{gather*} is an exterior algebra on $(g-1)$ generators of degree~$3$ for all but one exceptional $C^{\infty}$-type $\tilde{\tau}$ which occurs only when $(\Sigma,\tau)$ is a type~I curve of even genus. If the genus $g \geq 4$, then for the one exceptional type, $H^*(M_{\Lambda}(2,d,\tilde{\tau}); k)$ is not an exterior algebra, as its first non-zero Betti number occurs in degree~$2$.	
\end{Corollary}

Calculations by the author (to appear elsewhere) show that $H^*(M_{\Lambda}(2,d, \tilde{\tau}); {\mathbb Q})$ contains interesting information depending not only on $\tau$, but also on $\tilde{\tau}$. Corollary~\ref{BigCor} implies that the interesting parts of $H^*(M_{\Lambda}(2,d,\tilde{\tau}); {\mathbb Q})$ are generally not invariant under the action by $T_2$. This is surprising, because for complex moduli spaces the analogous action is trivial: this was considered by Atiyah and Bott to be the main result of the famous Harder--Narasimhan paper~\cite{HN} (see \cite[Section~9]{AB}). We conclude that to understand the topology of $M(r,d,\tilde{\tau})$, it is important to study the f\/ixed determinant moduli spaces $M_{\Lambda}(r,d,\tilde{\tau})$. This is the focus of ongoing work by the author.

\section{Preliminaries}

\subsection{Topological classif\/ications}\label{Review}

Let us begin by recalling the possibilities for a real structure on a Riemann surface. The nomenclature is that of \cite{BHH}. The possible structures are:
\begin{itemize}\itemsep=0pt
\item {Type 0 curves}: On these, the real structure has no f\/ixed points.
\item{Type I curves}: For these, the real structure $\tau$ has $a\leq (g+1)$ f\/ixed circles such that the complement of the real points, $\Sigma \setminus \Sigma^{\tau}$, is disconnected. Necessarily, $a \equiv g+1$ $({\rm mod}~2)$.
\item{Type II curves}: For these, the real structure $\tau$ has $a \leq g$ f\/ixed circles such that $\Sigma \setminus \Sigma^{\tau}$ is connected.
\end{itemize}
For all of these, one can write $\Sigma$ as the union of two copies of a surface with boundary $\Sigma_0$, with the identif\/ication taken along their boundaries, with $\tau$ interchanging the two copies. This is described in more detail in Section~\ref{constructtheclasspace}.

If $(E,\tilde{\tau})\rightarrow (\Sigma, \tau)$ is a Real $C^{\infty}$-vector bundle of rank $r$, then the f\/ixed point set $E^{\tilde{\tau}}$ forms an ordinary Real vector bundle over $\Sigma^{\tau}$, with f\/ibre ${\mathbb R}^r$. Since $\Sigma^{\tau}$ is a disjoint union of circles, the isomorphism type of $E^{\tilde{\tau}}$ is completely determined by the f\/irst Stiefel--Whitney class $w_1(E^{\tilde{\tau}})$. In particular, for each path component $S^1 \subseteq \Sigma^{\tau}$ we have
\begin{gather*} w_1\big(E^{\tilde{\tau}}|_{S^1}\big)\big(S^1\big) = \begin{cases} 0 & \text{if $E^{\tilde{\tau}}|_{S^1}$ is orientable (hence trivial)}, \\
1 & \text{if $E^{\tilde{\tau}}|_{S^1}$ is nonorientable (hence a M\"obius bundle)}. \end{cases} \end{gather*}
It follows then that
\begin{gather*} w_1\big(E^{\tilde{\tau}}\big)\big(\Sigma^{\tau}\big) = \begin{cases} 0 & \text{if $E^{\tilde{\tau}}$ is nonorientable on an even number of path components}, \\ 1 & \text{if $E^{\tilde{\tau}}$ is nonorientable on an odd number of path components}. \end{cases} \end{gather*}

The classif\/ication of topological Real/Quaternionic vector bundles over a real curve $(\Sigma, \tau)$ is as follows (see Propositions~4.1 and~4.2 of~\cite{BHH}).

\begin{Proposition}\label{arecbu}
Topological Real vector bundles $(E,\tilde{\tau})$ over a real curve $(\Sigma, \tau)$ are classified up to isomorphism by rank $r$, degree $d$ and Stiefel--Whitney class $w_1(E^{\tilde{\tau}}) \in H^1(\Sigma^{\tau};{\mathbb Z}_2)$ subject to the condition that \begin{gather*}d \equiv w_1\big(E^{\tilde{\tau}}\big)\big(\Sigma^{\tau}\big) \quad {\rm mod}~2. \end{gather*}
In particular, if the fixed point set $\Sigma^{\tau}$ is a union of $a\geq 1$ disjoint circles, then there are $2^{a-1}$ isomorphism classes of Real $C^{\infty}$-vector bundles over $(\Sigma,\tau)$ of any fixed rank and degree.

Quaternionic vector bundles are classified by rank $r$ and degree $d$, subject to the condition
\begin{gather*}
d \equiv r(g-1) \quad {\rm mod}~2
\end{gather*}
and that $\Sigma^{\tau} = \varnothing$ if $r$ is odd.
\end{Proposition}

\begin{Remark}\label{quatnotint}
It follows from Proposition \ref{arecbu} that a rank two Quaternionic vector bundle must have even degree. This justif\/ies our greater focus on Real bundles, since we are more interested in bundles with coprime rank and degree. It also follows that odd degree Real bundles can only occur over curves with real points, so that $w_1(E^{\tilde{\tau}})$ may be non-zero. So for odd degree bundles, we need only consider curves of type I or II.
\end{Remark}

\subsection{The complex Harder--Narasimhan stratif\/ication}\label{HNstratC}
A rank two holomorphic bundle over a curve ${\mathcal E} \rightarrow \Sigma$ is called \emph{semi-stable} if it does not contains any line sub-bundle of degree greater than $d/2$, where $d$ is the degree of ${\mathcal E}$. If ${\mathcal E}$ is not semi-stable, then we say it is \emph{unstable}. Each unstable rank two bundle contains a unique line subbundle of maximum degree called the \emph{SCSS line sub-bundle} (strongly contradicting semi-stability, see Harder--Narasimhan~\cite{HN}).

Fix a $C^{\infty}$-vector bundle $E$ of rank two and degree $d$ over $\Sigma$. Let ${\mathcal C} ={\mathcal C}(E)$ denote the space of Cauchy--Riemann operators on $E$. This is a contractible manifold modelled on a Sobolev completion of $\Omega^{0,1}(\Sigma,\operatorname{End}(E))$. The Harder--Narasimhan stratif\/ication decomposes ${\mathcal C}$ into f\/inite codimension submanifolds
\begin{gather}\label{HNstrateqn}
 {\mathcal C} = {\mathcal C}_{ss} \cup \bigg( \bigcup_{d_1 > d/2} {\mathcal C}_{d_1}\bigg).
\end{gather}

Here ${\mathcal C}_{ss}$ is the subset of Cauchy--Riemann operators giving semi-stable holomorphic bundles; it is the open stratum. The set ${\mathcal C}_{d_1}$ is the subset of Cauchy--Riemann operators determining unstable bundles with SCSS line sub-bundle of degree $d_1$; it is a locally closed submanifold of complex codimension $(2d_1-d+g-1)$ in ${\mathcal C}$.\footnote{By work of Atiyah--Bott~\cite{AB} and Daskalopoulos~\cite{D}, (\ref{HNstrateqn}) is also the Morse stratif\/ication induced by the Yang--Mills functional, but we won't use this fact.} The complex gauge group ${\mathcal G} = {\mathcal G}(E)$ acts naturally on~${\mathcal C}$, preserving the stratif\/ication. The subgroup ${\mathbb C}^* \leq {\mathcal G}$ acts trivially and the quotient $\overline{{\mathcal G}} = {\mathcal G}/{\mathbb C}^*$ acts ef\/fectively on ${\mathcal C}$. Because ${\mathcal C}$ is contractible, the homotopy quotient $ {\mathcal C}_{h\overline{{\mathcal G}}}= E\overline{{\mathcal G}} \times_{\overline{{\mathcal G}}} {\mathcal C}$ is a~model for the classifying space $B \overline{{\mathcal G}}$.
\begin{gather*} B\overline{{\mathcal G}} = {\mathcal C}_{h\overline{{\mathcal G}}}.\end{gather*}
The stratif\/ication (\ref{HNstrateqn}) descends to a stratif\/ication
\begin{gather}\label{HNstratquot}
 B\overline{{\mathcal G}} = ({\mathcal C}_{ss})_{{h\overline{{\mathcal G}}} } \cup \bigg( \bigcup_{d_1 > d/2} ({\mathcal C}_{d_1})_{{h\overline{{\mathcal G}}} } \bigg).
\end{gather}
The (topological) moduli stack of semistable, rank two, degree $d$ bundles on $\Sigma$ is the homotopy quotient
\begin{gather*} {\mathcal M}(2,d) = {\mathcal M}(E) = ({\mathcal C}_{ss})_{{h\overline{{\mathcal G}}} } .\end{gather*}
If $d$ is odd, then $\overline{{\mathcal G}}$ acts freely on ${\mathcal C}_{ss}$ and we may identify $ {\mathcal M}(2,d)$ with the coarse moduli space \begin{gather*}M(2,d) = {\mathcal C}_{ss}/\overline{{\mathcal G}},\end{gather*} which is a complex manifold of complex dimension $4g-3$ when $g>1$ .

Now suppose $(\Sigma, \tau)$ is a real curve. Choose a real or quaternionic structure $\tilde{\tau}$ on the $C^{\infty}$-vector bundle $E$. Let $\hat{{\mathcal G}}$ be the group of transformations of $E$ generated by ${\mathcal G}$ and $\tilde{\tau}$. Note that~$\hat{{\mathcal G}}$ is independent of the choice of~$\tilde{\tau}$, because for any other choice $\tilde{\tau}'$, the composition $\tilde{\tau} \tilde{\tau}' \in {\mathcal G}$, so we have an equality of cosets $\tilde{\tau} {\mathcal G} = \tilde{\tau}'{\mathcal G}$. The natural action of $\hat{{\mathcal G}}$ on ${\mathcal C}$ preserves the stratif\/ication~(\ref{HNstrateqn}). This descends to a residual action of ${\mathbb Z}_2 = \hat{{\mathcal G}}/{\mathcal G}$ on $B\overline{{\mathcal G}}$ which preserves the stratif\/ication~(\ref{HNstratquot}) and acts by anti-holomorphic involutions on the strata (we denote this involution by $\tau$ by abuse of notation). In particular, this means that the normal bundles of strata~(\ref{HNstratquot}) are Real vector bundles with respect to~$\tau$.

\subsection{The real Harder--Narasimhan stratif\/ication}\label{realstratsect}

Let $\tilde{\tau}$ denote a real or quaternionic structure on $E$ and let ${\mathcal C}^{\tilde{\tau}} ={\mathcal C}(E,\tilde{\tau}) \subset {\mathcal C}(E)$ denote the subspace of Cauchy--Riemann operators that are invariant under $\tilde{\tau}$. It was explained in \cite{B} that the Harder--Narasimhan stratif\/ication determines a stratif\/ication of ${\mathcal C}^{\tilde{\tau}}${\samepage
\begin{gather}\label{rHNstrat}
 {\mathcal C}^{\tilde{\tau}} = {\mathcal C}_{ss}^{\tilde{\tau}} \cup \bigg( \bigcup_{d_1 > d/2} {\mathcal C}_{d_1}^{\tilde{\tau}} \bigg).
 \end{gather}
where ${\mathcal C}^{\tilde{\tau}}_{d_1} = {\mathcal C}_{d_1} \cap {\mathcal C}^{\tilde{\tau}}$ is a locally closed submanifold of real codimension $(2d_1-d+g-1)$.}

Def\/ine the \emph{real/quaternionic gauge group} ${\mathcal G}^{\tilde{\tau}} = {\mathcal G}(E,\tilde{\tau})$, to be the group of gauge transformations of $E$ that commute with $\tilde{\tau}$. The subgroup of scalars ${\mathbb R}^*$ act trivially on ${\mathcal C}^{\tilde{\tau}}$ and the quotient $\overline{{\mathcal G}^{\tilde{\tau}}} = {\mathcal G}^{\tilde{\tau}}/{\mathbb R}^*$ acts ef\/fectively on ${\mathcal C}(E,\tilde{\tau})$ preserving the stratif\/ication (\ref{rHNstrat}). Since ${\mathcal C}^{\tilde{\tau}}$ is contractible, we have a homotopy quotient
\begin{gather*} B\overline{{\mathcal G}^{\tilde{\tau}}} = {\mathcal C}_{h\overline{{\mathcal G}^{\tilde{\tau}}}}^{\tilde{\tau}}.\end{gather*}
and (\ref{rHNstrat}) descends to a stratif\/ication
\begin{gather}\label{rHNstratquot} B\overline{{\mathcal G}^{\tilde{\tau}}} = ({\mathcal C}_{ss}^{\tilde{\tau}})_{h\overline{{\mathcal G}^{\tilde{\tau}}}} \cup \bigg( \bigcup_{d_1 > d/2} ({\mathcal C}_{d_1}^{\tilde{\tau}})_{h\overline{{\mathcal G}^{\tilde{\tau}}}} \bigg). \end{gather}

The (topological) moduli stack ${\mathcal M}(E,\tilde{\tau})$ is the homotopy quotient
\begin{gather*} {\mathcal M}(E,\tilde{\tau}) = \big({\mathcal C}_{ss}^{\tilde{\tau}}\big)_{_{h\overline{{\mathcal G}^{\tilde{\tau}}}}}. \end{gather*} If $d$ is odd, $\overline{{\mathcal G}^{\tilde{\tau}}}$ acts freely and we identify ${\mathcal M}(E,\tilde{\tau})$ with the orbit space \begin{gather*}M(E,\tilde{\tau}) = {\mathcal C}_{ss}^{\tilde{\tau}}/ {\mathcal G}^{\tilde{\tau}}\end{gather*} which is a compact manifold of real dimension $4g-3$ when $g>1$ that embeds in $M(2,d)$ as a~path component of the f\/ixed point set $M(2,d)^{\tau}$ (see~\cite{S12}).

In some cases, the higher strata are actually empty.

\begin{Theorem}\label{QuatDone}
Suppose that $(E,\tilde{\tau}) \rightarrow (\Sigma, \tau)$ is a rank two $C^{\infty}$-quaternionic bundle for which the fixed point set $\Sigma^{\tau}$ is non-empty. Then the natural map $ {\mathcal M}(E,\tilde{\tau}) \hookrightarrow B\overline{{\mathcal G}^{\tilde{\tau}}} $ is a homotopy equivalence.
\end{Theorem}

\begin{proof}
Suppose that ${\mathcal E}$ is an unstable rank two holomorphic bundle over $\Sigma$. If ${\mathcal E}$ were to admit a quaternionic structure lifting $\tau$, then this would restrict to a quaternionic structure on the SCSS line sub-bundle (see \cite[Section~2.2]{B}). But this contradicts Proposition~\ref{arecbu}, because $\Sigma^{\tau}$ is non-empty. It follows that every quaternionic lift of $\tau$ is semistable, so ${\mathcal C}^{\tilde{\tau}}_{ss} = {\mathcal C}^{\tilde{\tau}}$ and the result follows.
\end{proof}

We are interested in determining the orientability of the normal bundle in the stratif\/ica\-tion~(\ref{rHNstratquot}). The inclusion ${\mathcal C}^{\tilde{\tau}} \hookrightarrow {\mathcal C}$ respects the stratif\/ication, by construction, and is equivariant relative to the inclusion homomorphism ${\mathcal G}^{\tilde{\tau}} \hookrightarrow {\mathcal G}$, so it descends to a map
\begin{gather*} i\colon \ B \overline{{\mathcal G}^{\tilde{\tau}}} \rightarrow B\overline{{\mathcal G}}\end{gather*}
that respects the stratif\/ications (\ref{HNstratquot}) and (\ref{rHNstratquot}). The restriction of $i$ to a map between corresponding strata \begin{gather*}({\mathcal C}_{d_1}^{\tilde{\tau}})_{h\overline{{\mathcal G}^{\tilde{\tau}}}} \rightarrow ({\mathcal C}_{d_1})_{h\overline{{\mathcal G}}}\end{gather*} determines a homotopy equivalence between $({\mathcal C}_{d_1}^{\tilde{\tau}})_{h\overline{{\mathcal G}^{\tilde{\tau}}}}$ and the union of those path components of the f\/ixed point set $(({\mathcal C}_{d_1})_{h\overline{{\mathcal G}}})^{\tau}$ corresponding to $\tilde{\tau}$. This identif\/ies the normal bundles to strata in~(\ref{rHNstratquot}) with the real points of the pull-backs of the normal bundles of (\ref{HNstratquot}) equipped with the real structure from $\hat{{\mathcal G}}/{\mathcal G} = {\mathbb Z}_2$. Thus we can determine orientability of normal bundles in~(\ref{rHNstratquot}) by studying the normal bundles of strata~(\ref{HNstratquot}) with real structure~$\tau$.

\section{Orientability of normal bundles}\label{ornormbund}

The goal of this section is to prove the following.

\begin{Proposition}\label{orientableProp}
The normal bundles of the stratification \eqref{rHNstratquot} are all nonorientable if any of the following holds:
\begin{enumerate}\itemsep=0pt
\item[$1.$] The degree of the bundle and the genus of the curve are of the same parity.
\item[$2.$] The type of the curve is I, and the Stiefel--Whitney class of the bundle vanishes on at least one component of the invariant curve $\Sigma^{\tau}$.
\item[$3.$] The type of the curve is II.
\end{enumerate}
Conversely, if $\Sigma^{\tau} \neq \varnothing$ and none of the three conditions hold, then the normal bundles are all orientable.
\end{Proposition}

To understand the normal bundles of our real strata $({\mathcal C}_{d_1}^{\tilde{\tau}})_{h\overline{{\mathcal G}^{\tilde{\tau}}}}$, we identify them as ${\mathbb Z}_2$-f\/ixed point sets of the normal bundles for the complex strata $ ({\mathcal C}_{d_1})_{{h\overline{{\mathcal G}}} }$. This approach considers all the $C^{\infty}$-types for the lifts $\tilde{\tau}$ of $\tau$ at once, and we must be careful to identify which path components corresponds to which lifts $\tilde{\tau}$ .

For unstable complex strata, we have a homotopy equivalence
\begin{gather}\label{stratquot2} ({\mathcal C}_{d_1})_{h\overline{{\mathcal G}}} = \operatorname{Pic}^{d_1}(\Sigma) \times \operatorname{Pic}^{d_2}(\Sigma) \times {\mathbb C} P^{\infty} ,
\end{gather}
where $d_2 = d-d_1$. This can be explained as follows. Choose a decomposition
\begin{gather}\label{decompintolines} E = L_1 \oplus L_2
\end{gather} into a sum of $C^{\infty}$-line bundles $L_1$, $L_2$ of degrees $d_1$, $d_2$ respectively. Let ${\mathcal C}(L_i)$ denote the space of Cauchy--Riemann operators on~$L_i$. The gauge group ${\mathcal G}(L_i)$ acts naturally on ${\mathcal C}(L_i)$ with orbit space \begin{gather*}{\mathcal C}(L_i)/{\mathcal G}(L_i) = \operatorname{Pic}^{d_i}(\Sigma).\end{gather*}
This action is not free because the constant scalar transformations act trivially. Choose a base point $p_0 \in \Sigma$ and denote ${\mathcal G}_{\rm bas}(L_i) \subseteq {\mathcal G}(L_i)$ the subgroup of gauge transformations that act trivially on the f\/ibre above $p_0$. We have an internal direct product decomposition
\begin{gather*} {\mathcal G}(L_i) = {\mathcal G}_{\rm bas}(L_i) \times {\mathbb C}^* , \end{gather*}
where the subgroup of scalar transformations ${\mathbb C}^*$ acts trivially on ${\mathcal C}(L_i)$ and ${\mathcal G}_{\rm bas}(L_i)$ acts freely on ${\mathcal C}(L_i)$. The decomposition (\ref{decompintolines}) induces morphisms ${\mathcal C}(L_1) \times {\mathcal C}(L_2) \hookrightarrow {\mathcal C}(E)$ and \begin{gather*} {\mathcal G}_{\rm bas}(L_1) \times {\mathcal G}(L_2) \cong {\mathcal G}_{\rm bas}(L_1) \times {\mathcal G}_{\rm bas}(L_2) \times {\mathbb C}^* \hookrightarrow \overline{{\mathcal G}}, \end{gather*} and determines a homotopy equivalence of homotopy quotients~\cite[Section~7]{AB}
\begin{gather*}({\mathcal C}(L_1) \times {\mathcal C}(L_2))_{h ({\mathcal G}_{\rm bas}(L_1) \times {\mathcal G}_{\rm bas}(L_2) \times {\mathbb C}^*)} \cong ({\mathcal C}_{d_1})_{h\overline{{\mathcal G}}}. \end{gather*}
The subgroup ${\mathcal G}_{\rm bas}(L_1) \times {\mathcal G}_{\rm bas}(L_2)$ acts freely and ${\mathbb C}^*$ acts trivially, so the homotopy quotient~(\ref{stratquot2}) may be identif\/ied with the orbit space of ${\mathcal C}(L_1) \times {\mathcal C}(L_2) \times E{\mathbb C}^* $ under the product action by the group ${\mathcal G}_{\rm bas}(L_1) \times {\mathcal G}_{\rm bas}(L_2) \times {\mathbb C}^*$. That is,
\begin{gather*} ({\mathcal C}_{d_1})_{h\overline{{\mathcal G}}} \cong {\mathcal C}(L_1)/ {\mathcal G}_{\rm bas}(L_1) \times {\mathcal C}(L_2)/ {\mathcal G}_{\rm bas}(L_2) \times B {\mathbb C}^*\\
\hphantom{({\mathcal C}_{d_1})_{h\overline{{\mathcal G}}} }{} \cong \operatorname{Pic}^{d_1}(\Sigma) \times \operatorname{Pic}^{d_2}(\Sigma) \times {\mathbb C} P^{\infty}.
\end{gather*}

The action of $\tau$ on (\ref{stratquot2}) is a product action on each of the three factors. In terms of divisor classes $[D_i] \in \operatorname{Pic}^{d_i}(\Sigma)$ and a projective point $[v] \in {\mathbb C} P^{\infty}$ the action sends $ ([D_1], [D_2], [v])$ to $([\tau(D_1)], [\tau(D_2)], [\overline{v}])$ (this involution on the $\operatorname{Pic}^{d_i}(\Sigma)$ factors was studied by Gross--Harris~\cite{GH2}). The f\/ixed point set, denoted
\begin{gather}\label{strataformula}
\big(\operatorname{Pic}^{d_1}(\Sigma) \times \operatorname{Pic}^{d - d_1}(\Sigma) \times {\mathbb C} P^{\infty}\big)^{{\mathbb Z}_2} = \operatorname{Pic}^{d_1}(\Sigma)^{\tau} \times \operatorname{Pic}^{d - d_1}(\Sigma)^{\tau} \times {\mathbb R} P^{\infty}
\end{gather}
is a union of path components, each homeomorphic to
\begin{gather*} \big(S^1\big)^g \times \big(S^1\big)^g \times {\mathbb R} P^{\infty}.\end{gather*}
The dif\/ferent components correspond to the dif\/ferent $C^{\infty}$-types of the lift $\tilde{\tau}$ of the real struc\-tu\-re~$\tau$ to the bundle and of restrictions of $\tilde{\tau}$ to the SCSS line sub-bundle. These $C^{\infty}$-types are classif\/ied by Stiefel--Whitney classes according to Section~\ref{Review}. If $\Sigma^{\tau}$ has $a \geq 1$ components, then the f\/ixed point set~(\ref{strataformula}) has $2^{2a-2}$ components.

The normal bundle $N$ of $({\mathcal C}_{d_1})_{h\overline{{\mathcal G}}}$ is a complex vector bundle constructed in two stages as follows (for example, see the proof of Lemma~2 in~\cite{K}). Consider the vector bundle $N''$ over the Banach manifold ${\mathcal C}(L_1) \times {\mathcal C}(L_2)$ with f\/ibres given by sheaf cohomology groups \begin{gather*}N''_{(\bar{\partial}_1, \bar{\partial}_2)} = H^1\big( L_1^* \otimes L_2,\bar{\partial}_1^* \otimes \bar{\partial_2}\big). \end{gather*}
Since $L_1^* \otimes L_2$ has negative degree, it admits no holomorphic sections for any choice of Cauchy--Riemann operator. By Riemann--Roch, it follows that $N'' $ is a vector bundle of rank $(2d_1-d +g-1)$. The action of ${\mathcal G}_{\rm bas}(L_1) \times {\mathcal G}(L_2)$ lifts naturally to $N''$. The subgroup ${\mathcal G}_{\rm bas}(L_1) \times {\mathcal G}_{\rm bas}(L_2)$ acts freely, and the quotient yields a holomorphic vector bundle \begin{gather*}N' = N'' / {\mathcal G}_{\rm bas}(L_1) \times {\mathcal G}_{\rm bas}(L_2)\end{gather*} over $ \operatorname{Pic}^{d_1}(\Sigma) \times \operatorname{Pic}^{d_2}(\Sigma)$. The subgroup ${\mathbb C}^*$ acts with weight one on the f\/ibres of $N''$, hence also on $N'$, giving rise to a vector bundle \begin{gather*}N = N'_{h {\mathbb C}^*}= N' \times_{{\mathbb C}^*} E{\mathbb C}^*\end{gather*} over $\operatorname{Pic}^{d_1}(\Sigma) \times \operatorname{Pic}^{d_2}(\Sigma) \times {\mathbb C} P^{\infty}.$

Recall that $ N$ has a real structure $\tau$ def\/ined in Section~\ref{HNstratC}.

\begin{Lemma}\label{g+deven}
If $g+d$ is even, then the normal bundle $N^{\tau}$ is nonorientable on every path component of $({\mathcal C}_{d_1})_{h\overline{{\mathcal G}}}^{\tau}$.
\end{Lemma}

\begin{proof}
Let $x$ be a $\tau$-f\/ixed point in $\operatorname{Pic}^{d_1}(\Sigma) \times \operatorname{Pic}^{d_2}(\Sigma)$ and consider the restriction of $N^{\tau}$ to $\{x\} \times B {\mathbb R}^*$. Because ${\mathbb C}^*$ acts by scalar multiplication of weight one on $N'$, we have an isomorphism \begin{gather*}N^{\tau}|_{\{x\} \times B {\mathbb R}^*} = N'^{\tau'}_x \times_{{\mathbb R}^*} E {\mathbb R}^*, \end{gather*} which is isomorphic to the Whitney sum of $\operatorname{rank}(N')$ copies of the tautological bundle over $B {\mathbb R}^* = {\mathbb R} P^{\infty}$. This is nonorientable if and only if $\operatorname{rank}(N')$ is odd, which is true if and only if $g+d$ is even.
\end{proof}

The involution $\tau$ on $N$ lifts to an involution $\tau'$ of $N'$ by identifying $N'$ with the restriction of~$N$ to $\operatorname{Pic}^{d_1}(\Sigma) \times \operatorname{Pic}^{d_2}(\Sigma) \times \{[v]\}$ for some f\/ixed point $\{[v]\} \in {\mathbb R} P^{\infty} \subset {\mathbb C} P^{\infty}$. If the degree $d$ and the genus $g$ have dif\/ferent parity, then $N^{\tau}$ is nonorientable if and only if $N'^{\tau}$ is nonorientable, or equivalently the f\/irst Stiefel--Whitney class $w_1(N'^{\tau})$ does not vanish. The equalities{\samepage \begin{gather*}w_1(N'^{\tau}) = w_1(\det(N'^{\tau})) = w_1(\det(N')^{\tau})\end{gather*} permits us to work with the determinant line bundle~$\det(N')$.}

The real structure on $N'$ induces one on $\det(N')$. Since $\det(N')$ is a holomorphic line bundle, this real structure is unique up to composition with a unit scalar (an analogue of Schur's lemma, see~\cite{BHH} or~\cite{OT}), so the $C^{\infty}$-type of the lift is unique. Thus if we construct any real structure on $\det(N')$, it must coincide with the one induced by~$\tau$ up to $C^{\infty}$-isomorphism. Such a real structure has been carefully studied by Okonek--Teleman~\cite{OT} (see also Cretois~\cite{C}).

\begin{Lemma}\label{themapphi}
Let
\begin{gather*}\phi\colon \ \operatorname{Pic}^{d_1}(\Sigma) \times \operatorname{Pic}^{d_2}(\Sigma) \rightarrow \operatorname{Pic}^{d_2-d_1}(\Sigma)\end{gather*}
be the map on divisor classes $ \phi([D_1], [D_2]) = [D_2-D_1]$. Then there is an isomorphism of holomorphic line bundles \begin{gather*} \det(N') \cong \phi^*( {\mathcal L}), \end{gather*}
where ${\mathcal L} := \operatorname{det \, ind} \delta_{p_0}^L$ is the determinant line bundle considered by Okonek--Teleman {\rm \cite[Section~1]{OT}}.
\end{Lemma}

\begin{proof}
Consider the bundle $V''$ over ${\mathcal C}(L_1^*\otimes L_2)$ with f\/ibres $V''_{\bar{\partial}} = H^1(L_1^* \otimes L_2, \bar{\partial})$. Clearly~$N''$ is the pull-back of~$V''$ under the map ${\mathcal C}(L_1) \times {\mathcal C}(L_2) \rightarrow {\mathcal C}(L_1^* \otimes L_2)$ that sends $(\bar{\partial}_1, \bar{\partial}_2)$ to~$\bar{\partial}_1^*\otimes \bar{\partial}_2$. This map is equivariant with respect to based gauge groups, and descends to $\phi$ as a~map between orbit spaces. Thus the quotient bundle~$N'$ is the pull-back of the quotient bundle $V' = V''/ {\mathcal G}(L_1^* \otimes L_2)$ over $\operatorname{Pic}^{d_2-d_1}(\Sigma)$. The determinant~$\det(V')$ is the determinant line bundle~${\mathcal L}$ considered by Okonek--Teleman.
\end{proof}

As explained in \cite[Section~1]{OT}, ${\mathcal L}$ is isomorphic to the line bundle obtained by translating the geometric theta divisor $\Theta \subset \operatorname{Pic}^{g-1}(\Sigma)$ by $[(d_2-d_1-g+1)p_0] \in \operatorname{Pic}(\Sigma)$. If we choose a real base point $p_0 \in \Sigma^{\tau}$, then this divisor is sent to itself by $\tau$, inducing a real structure on ${\mathcal L}$. The Stiefel--Whitney class of ${\mathcal L}^{\tau}$ was calculated in Theorem~4.15 of~\cite{OT} and the following lemma is a~direct corollary.

\begin{Lemma}\label{orientability}
Suppose that $(\Sigma,\tau)$ is a real curve with real base point $p_0 \in \Sigma^{\tau}$, and suppose that $g +d$ is odd.
\begin{itemize}\itemsep=0pt
\item[$(i)$]If $(\Sigma,\tau)$ is type I, then $w_1( {\mathcal L}^{\tau})$ vanishes only on the path component $M(L, \tilde{\tau}_L) {\subseteq}\! \operatorname{Pic}^{d_2{-}d_1}\!(\Sigma)^{\tau}$ for which the restrictions $L^{\tilde{\tau}_L}|_{S^1}$ are nonorientable for all real circles $S^1 \subseteq \Sigma^{\tilde{\tau}}$.	
\item[$(ii)$] If $(\Sigma,\tau)$ is type~II, then $w_1( {\mathcal L} ^{\tau})$ does not vanish on any path component of $\operatorname{Pic}^{d_2-d_1}(\Sigma)^{\tau}$.
\end{itemize}
\end{Lemma}

Finally, we identify which path components correspond to which real structure on~$E$.

\begin{Lemma}\label{orientSWstuff}
Let $(E,\tilde{\tau}) \rightarrow (\Sigma,\tau)$ be a rank two $C^{\infty}$-real bundle. The map $\phi$ sends all path components of $\operatorname{Pic}^{d_1}(\Sigma)^{\tau} \times \operatorname{Pic}^{d_2}(\Sigma)^{\tau}$ corresponding to an unstable stratum of ${\mathcal C}^{\tilde{\tau}}$ to a~unique path component of $M(L, \tilde{\tau}_L) \subseteq \operatorname{Pic}^{d_2-d_1}(\Sigma)^{\tau}$ for which the Stiefel--Whitney classes agree:{\samepage \begin{gather*}w_1\big( L^{\tilde{\tau}_{L}}\big) = w_1\big(E^{\tilde{\tau}}\big)\end{gather*}
in $H^1(\Sigma^{\tau}; {\mathbb Z}/2)$.}
\end{Lemma}

\begin{proof}
If $(E,\tilde{\tau})$ decomposes as a sum of real bundles $(L_1 \oplus L_2, \tilde{\tau}_1 \oplus \tilde{\tau}_2)$, then
 \begin{gather*}
 w_1\big(E^{\tilde{\tau}} \big) = w_1\big((L_1 \otimes L_2)^{(\tilde{\tau}_1 \otimes \tilde{\tau}_2)}\big) = w_1\big((L_1^* \otimes L_2)^{(\tilde{\tau}_1^* \otimes \tilde{\tau}_2)}\big).\tag*{\qed}
 \end{gather*}
\renewcommand{\qed}{}
\end{proof}

\begin{proof}[Proof of Proposition~\ref{orientableProp}]
Suf\/f\/icient condition 1 follows from Lemma~\ref{g+deven}.

For suf\/f\/icient conditions 2 and 3, consider the restriction of the map $\phi$ in Lemma~\ref{themapphi} to path components of $\tau$ f\/ixed point sets.
\begin{gather*} \phi'\colon \ M(1,d_1, \tilde{\tau}_1) \times M(1,d_2,\tilde{\tau}_2) \rightarrow M(1,d_2-d_1, \tilde{\tau}_1^* \otimes \tilde{\tau}_2).\end{gather*}
Choose a f\/ixed element $[D_1] \in M(1,d_1, \tilde{\tau}_1)$. Then the map
\begin{gather*}\psi\colon \ M(1,d_2-d_1, \tilde{\tau}_1^* \otimes \tilde{\tau}_2) \rightarrow M(1,d_1, \tilde{\tau}_1) \times M(1,d_2,\tilde{\tau}_2)\end{gather*}
sending $\psi( [D]) = ( [D_1], [D +D_1])$ is a left inverse of $\phi'$ (i.e., $\phi' \circ \psi$ is the identity map on $M(1,d_2-d_1, \tilde{\tau}_1^* \otimes \tilde{\tau}_2 )$. It follows then that $\phi'$ induces an injection on cohomology. In particular, the pullback of a nonorientable vector bundle by $\phi'$ must be nonorientable. The result now follows from Lemmas~\ref{orientability} and~\ref{orientSWstuff}.

For the converse statement, we have a type I curve equipped with a Real bundle $(E,\tilde{\tau})$ such that the genus and degree have opposite parity and $w_1(E^{\tilde{\tau}})$ is non-vanishing on all path components of~$\Sigma^{\tau}$. From Lemmas~\ref{orientability} and~\ref{orientSWstuff}, we f\/ind that $N'^{\tau}$ is the pullback of an orientable bundle, hence is orientable. Finally, from the proof of Lemma~\ref{g+deven}, the quotient $N^{\tau} = N'^{\tau}/{\mathbb R}^*$ is also orientable.
\end{proof}

\section{Proof of Theorem \ref{isomcoh}}

The (path components of) unstable strata of the real Harder--Narasimhan stratif\/ication (\ref{rHNstratquot}) are homotopy equivalent to $X := (S^1)^{2g} \times {\mathbb R} P^{\infty} $. The contribution of that stratum into Morse theory is through the relative cohomology groups $H^*(N, N_0)$ where $N \rightarrow X$ is an ${\mathbb R}^n$-vector bundle (the normal bundle) and $N_0 \subset N$ is the complement of the zero section. If $N$ is orientable then we have the Thom isomorphism \begin{gather*}H^*(N, N_0) \cong H^{*-n}(X),\end{gather*}
but if $N$ is not orientable we instead get the following.

\begin{Proposition}\label{iszero}
Let $X := (S^1)^{2g} \times {\mathbb R} P^{\infty} $ and $k$ a field of characteristic $ \neq 2$. If $N \rightarrow X$ is a~nonorientable real vector bundle, then $H^*(N, N_0;k) = 0$.
\end{Proposition}

\begin{proof}
Suppose that $N \rightarrow X$ is a nonorientable vector bundle of rank $n+1$. Using the long exact sequence of the pair, it is equivalent to show that the inclusion induced map
\begin{gather}\label{toproveiso}
H^*(X) = H^*(N) \rightarrow H^*(N_0)
\end{gather} is an isomorphism.

The cohomology of the f\/ibre $H^*({\mathbb R}^{n+1} \setminus 0) \cong H^*(S^n)$ is $k$ in degree $0$ and $n$ and is zero otherwise. Furthermore, the action of $\pi_1(X)$ on $H^0(S^n)$ is trivial and on $H^n(S^n)$ factors through a non-trivial homomorphism $ \rho\colon \pi_1(X) \rightarrow {\mathbb Z}/2$ because~$N$ is nonorientable. Thus if we denote by~$k_{\rho}$ the locally constant $k$-sheaf twisted by $\rho$, then the Serre spectral sequence attached to the f\/ibre bundle $N_0 \rightarrow X $
has $E_2$-page satisfying $E_2^{0,q} \cong H^q(X;k)$, $E_2^{n,q} \cong H^q(X;k_{\rho})$ and $E_2^{p,q} =0$ for $p\neq 0, n$. If we prove that $H^*(X;k_{\rho})=0$, then the spectral sequence collapses and~(\ref{toproveiso}) is an isomorphism.

Let $\tilde{X} \rightarrow X$ denote the double cover of $X$ def\/ined by $ \rho\colon \pi_1(X) \rightarrow {\mathbb Z}/2$. Because we are working in a characteristic other than two, the transfer map def\/ines an isomorphism
\begin{gather}\label{contro}
H^*(\tilde{X};k) \cong H^*(X;k) \oplus H^*(X; k_{\rho}),
\end{gather}
where the direct sum decomposition is into the $\pm 1$-eigenspaces under the action by the deck transformation group ${\mathbb Z}/2$.

Since $X = K({\mathbb Z}^{2g}\times {\mathbb Z}_2, 1)$ is an Eilenberg--MacLane space, $\tilde{X} = K(\Gamma, 1)$ for an index two subgroup $ \Gamma \subset {\mathbb Z}^{2g}\times {\mathbb Z}_2$ which by the classif\/ication of f\/initely generated abelian groups is isomorphic either to $ {\mathbb Z}^{2g} \times {\mathbb Z}_2$ or $ {\mathbb Z}^{2g}$. In either case, \begin{gather*}H^*(\tilde{X};k) \cong H^*(K(\Gamma,1);k) \cong H^*(X; k), \end{gather*} which combined with~(\ref{contro}) implies that $H^*(X; k_{\rho})=0$.
\end{proof}

\begin{proof}[Proof of Theorem \ref{isomcoh}]
Suppose that $(E,\tilde{\tau})$ is a rank two, real $C^{\infty}$-vector bundle over a real curve $(\Sigma, \tau)$. Using the stratif\/ication~(\ref{rHNstratquot}), we construct the f\/iltration
\begin{gather*} \big({\mathcal C}_{ss}^{\tilde{\tau}}\big)_{h\overline{{\mathcal G}^{\tilde{\tau}}}} = Y_0 \subseteq Y_1 \subseteq \dots \subseteq Y_{\infty} = B\overline{{\mathcal G}^{\tilde{\tau}}}, \end{gather*}
where \begin{gather*} Y_i = \big({\mathcal C}_{ss}^{\tilde{\tau}}\big)_{h\overline{{\mathcal G}^{\tilde{\tau}}}} \cup \bigg( \bigcup_{ d/2+ i \geq d_1} ({\mathcal C}_{d_1}^{\tilde{\tau}})_{h\overline{{\mathcal G}^{\tilde{\tau}}}} \bigg).\end{gather*}
By excision, for each $i$, \begin{gather*}H^*(Y_i, Y_{i-1};k) \cong H^*(N_i, (N_i)_0; k),\end{gather*}
where $N_i$ is the normal bundle of an unstable stratum. If the normal bundles of all positive codimension strata in the real Harder--Narasimhan stratif\/ication are nonorientable, then it follows from Proposition~\ref{iszero} that $H^*(Y_i, Y_{i-1};k) = 0$ and thus that inclusion induces an isomorphism \begin{gather*}H^*(Y_{i-1};k) \cong H^*(Y_{i};k). \end{gather*}
Since this holds for all $i\geq 1$, it follows by induction that
\begin{gather*} H^*_{\overline{{\mathcal G}^{\tilde{\tau}}}}\big( {\mathcal C}_{ss}^{\tilde{\tau}};k\big) = H^*(Y_0;k) \cong H^*(Y_{\infty};k) = H^*\big(B\overline{{\mathcal G}^{\tilde{\tau}}};k\big). \end{gather*}
If $E$ has odd degree, the action of $\overline{{\mathcal G}^{\tilde{\tau}}}$ is free, so \begin{gather*}H^*_{\overline{{\mathcal G}^{\tilde{\tau}}}}\big( {\mathcal C}_{ss}^{\tilde{\tau}};k\big) \cong H^*\big( {\mathcal C}_{ss}^{\tilde{\tau}}/\overline{{\mathcal G}^{\tilde{\tau}}};k\big) = H^*(M(E,\tilde{\tau});k).\end{gather*}

Finally, we relate equivariant cohomology with respect to ${\mathcal G}^{\tilde{\tau}} $ and $\overline{{\mathcal G}^{\tilde{\tau}}} = {\mathcal G}^{\tilde{\tau}}/{\mathbb R}^*$. Given any ${\mathcal G}^{\tilde{\tau}}$-space $X$ on which the subgroup ${\mathbb R}^*$ acts trivially, the natural map on Borel constructions
\begin{gather*} E{\mathcal G}^{\tilde{\tau}} \times_{{\mathcal G}^{\tilde{\tau}}} X \rightarrow E\overline{{\mathcal G}^{\tilde{\tau}}} \times_{\overline{{\mathcal G}^{\tilde{\tau}}}} X \end{gather*}
has homotopy f\/ibre $B{\mathbb R}^* = {\mathbb R} P^{\infty}$. Since $H^*(B{\mathbb R}^*;k)$ is acyclic for coef\/f\/icient f\/ields $k$ of characteristic not equal to~$2$, the Serre spectral sequence is trivial, yielding the isomorphism
\begin{gather*}H^*_{{\mathcal G}^{\tilde{\tau}}}( X;k) = H^*\big( E{\mathcal G}^{\tilde{\tau}} \times_{{\mathcal G}^{\tilde{\tau}}} X; k\big) \cong H^*\big( E\overline{{\mathcal G}^{\tilde{\tau}}} \times_{\overline{{\mathcal G}^{\tilde{\tau}}}} X; k\big) = H^*_{\overline{{\mathcal G}^{\tilde{\tau}}}}( X;k) . \end{gather*}
Applying this isomorphism when $X = {\mathcal C}_{ss}^{\tilde{\tau}}$ and when $X$ is a point completes the proof.
\end{proof}

\section{Real gauge groups}\label{SScalcReal}

In this section, we compute the cohomology ring of the classifying space $B{\mathcal G}^{{\tilde{\tau}}}$. Recall that the Poincar\'e series of a space $X$ is the generating function for its Betti numbers \begin{gather*} P_t(X) = \sum_{i=0}^{\infty} b_it^i, \end{gather*}
where $b_i = \dim(H^i(X;k))$. The goal of this section is to prove the following.

\begin{Theorem}\label{bgring}
Let $(\Sigma, \tau)$ be a real curve of genus $g$ with real points $($i.e., $\Sigma^{\tau}\neq \varnothing)$. Let ${\mathcal G}^{\tilde{\tau}} = {\mathcal G}(E,\tilde{\tau})$ be the real gauge group of a~rank~$2$ Real bundle $(E,\tilde{\tau})$ over $(\Sigma, \tau)$, and let $k$ be a~coefficient field $k$ of odd or zero characteristic. Then $H^*(B{\mathcal G}^{{\tilde{\tau}}}; k)$ is an exterior algebra with Poincar\'e series \begin{gather*}P_t(B{\mathcal G}^{\tilde{\tau}}) =(1+t)^g\big(1+t^3\big)^{g-1},\end{gather*}
except in the following two special cases:
\begin{enumerate}\itemsep=0pt
	\item[$1.$] If the real structure on the curve is of type I and the restriction of $E^{{\tilde{\tau}}}$ to each component of $\Sigma^{\tau}$ is nonorientable, then $H^*(B{\mathcal G}^{\tilde{\tau}};k)$ is a free, graded commutative algebra with $P_t(B{\mathcal G}^{{\tilde{\tau}}}) = (1+t)^g(1+t^3)^{g}/(1-t^2).$
	
	\item[$2.$] If the restriction of $E^{{\tilde{\tau}}}$ to each component of $\Sigma^{\tau}$ is orientable, then $P_t(B{\mathcal G}^{{\tilde{\tau}}}) = (1+t)^g(1+t^3)^{g}/(1-t^4)$. Furthermore, if $k$ has characteristic zero, then $H^*(B{\mathcal G}^{\tilde{\tau}};k)$ is a free, graded commutative algebra.
\end{enumerate}
\end{Theorem}

\subsection{Constructing the classifying space}\label{constructtheclasspace}

Let $(\Sigma, \tau)$ be a real curve of genus $g$ and let $ (E,{\tilde{\tau}}) \rightarrow (\Sigma,\tau)$ be a rank two Real $C^{\infty}$-vector bundle. In this subsection, we construct the classifying space $B{\mathcal G}(E,\tilde{\tau})$ as a homotopy pull-back. Instead of working with ${\mathcal G}(E,\tilde{\tau})$ directly, we work with the subgroup of unitary gauge transformations $\mathcal{U}(E,\tilde{\tau})$. The inclusion $\mathcal{U}(E,\tilde{\tau}) \hookrightarrow {\mathcal G}(E,\tilde{\tau})$ is a homotopy equivalence, so they are interchangeable for our purposes.

Up to homeomorphism, every real curve $(\Sigma, \tau)$ can be constructed as follows. Let $\Sigma_0$ be a~genus $\hat{g}$ surface with~$n$ boundary components, such that $2\hat{g}+n-1 = g$. Construct $\Sigma$ by taking two copies of~$\Sigma_0$ with opposite orientations, and gluing them together along their boundaries, attaching $a \leq n$ boundary circles to their counterpart using the identity map, and attaching the rest using the antipodal map. The involution $\tau$ simply transposes these two copies of $\Sigma_0$. The resulting topological real curve has~$a$ f\/ixed point circles. We get a type~0 curve if $a=0$, a type~I curve if $a = n$, and a~type~II curve if $0< a< n$.

Given a Real bundle $(E,\tilde{\tau})$ over $(\Sigma,\tau)$, the unitary gauge symmetries that commute with $\tilde{\tau}$ (elements of $\mathcal{U}(E,{\tilde{\tau}})$) are determined by their restriction to $\Sigma_0= \Sigma(\hat g,n)$. On the boundary circles, the gauge symmetries restrict to transformations of three types:
\begin{itemize}\itemsep=0pt
\item[(a)] If $\tau|_{S^1}$ is the identity, and the restriction of $E^{\tilde{\tau}}$ to the circle is orientable, the gauge symmetries are maps from $S^1$ to $O(2)$.
\item[(b)] If $\tau|_{S^1}$ is the identity, and the restriction of $E^{\tilde{\tau}}$ to the circle is nonorientable, the gauge symmetries are those of the M\"obius ${\mathbb R}^2$-bundle.
\item[(c)] If $\tau|_{S^1}$ is a rotation by a half turn, our gauge symmetries satisfy $ g(\theta) = \overline{g(\theta + \pi)}$ where the bar means entry-wise complex conjugation.
\end{itemize}

In each case, the restrictions of the gauge transformations to the boundaries are the so called \emph{real loop groups} introduced in~\cite{B}. Let $LU_2^{{\tilde{\tau}}_i}$ denote the real loop group over the $i$th circle of the boundary of $\Sigma_0= \Sigma(\hat g,n)$, with def\/initions varying from circle to circle, according to the behaviour of $\tilde{\tau}$.

As explained in \cite[Section~6.1]{B}, the unitary real gauge group $\mathcal{U}(E,\tilde{\tau})$ can be constructed up to isomorphism as the pull-back ${\mathcal G}(\hat{g}, n; {\tilde{\tau}}_1,\dots,{\tilde{\tau}}_n)$ of topological groups
\begin{gather}\label{aoeurclorce1}
\begin{split}&
\xymatrix{ {\mathcal G}(\hat{g}, n; {\tilde{\tau}}_1,\dots,{\tilde{\tau}}_n) \ar[r] \ar[d] & \operatorname{Maps}_0(\Sigma(\hat{g},n), U_2) \ar[d]^{\pi}\\
\prod\limits_{i=1}^n LU_2^{{\tilde{\tau}}_i} \ar[r]^{\iota} & \prod\limits_{i=1}^n L_0U_2. }
\end{split}
\end{gather}
In this diagram
\begin{itemize}\itemsep=0pt
\item $\operatorname{Maps}_0(\Sigma(\hat{g}, n), U_2)$ is the group of maps from $\Sigma(\hat{g}, n)$ to $U_2$ that send each boundary circle of $\Sigma(\hat{g}, n)$ to a contractible loop in~$U_2$.
\item $L_0U_2$ is the group of contractible loops $f\colon S^1 \rightarrow U_2$ and $\pi$ is def\/ined by restriction from $\Sigma(\hat{g}, n)$ to $\partial \Sigma(\hat{g}, n)$.
\item The $LU_2^{{\tilde{\tau}}_i}$ are the real loop groups considered above, and $\iota$ is the product of inclusions $LU_2^{{\tilde{\tau}}_i} \hookrightarrow L_0U_2$.
\end{itemize}

 Applying the classifying space functor yields a homotopy pull-back square
\begin{gather}\label{aoeurclorce2}
\begin{split}
& \xymatrix{ B{\mathcal G}(\hat{g}, n; {\tilde{\tau}}_1,\dots,{\tilde{\tau}}_n) \ar[r] \ar[d] & \operatorname{BMaps}_0(\Sigma(\hat{g},n), U_2) \ar[d]^{B\pi}\\
\prod\limits_{i=1}^n BLU_2^{{\tilde{\tau}}_i} \ar[r]^{B\iota} & \prod\limits_{i=1}^n BL_0U_2. }
\end{split}
\end{gather}

Our strategy is to calculate $H^*(B{\mathcal G}(\hat{g}, n; {\tilde{\tau}}_1,\dots,{\tilde{\tau}}_n))$ using the Eilenberg--Moore spectral sequence of diagram~(\ref{aoeurclorce2}). We f\/irst need to understand the induced cohomology morphisms shown below
\begin{gather}\label{aoeurclorce...}
\begin{split}
& \xymatrix{ & H^*(\operatorname{BMaps}_0(\Sigma(\hat{g}, n), U_2)) \\
\bigotimes\limits_{i=1}^n H^*(BLU_2^{{\tilde{\tau}}_i}) & \ar[l]_{B\iota^*} \bigotimes\limits_{i=1}^n H^*(BL_0U_2) \ar[u]_{B\pi^*}. }
\end{split}
\end{gather}

Part of this was already calculated in \cite{B}.

\begin{Lemma}\label{Surfaceboundaries}
The map $B\pi^*$ fits into a commutative diagram
\begin{gather*}
\xymatrix{ \ar[r]^{\cong} A \otimes \frac{\Lambda(b_1,\dots,b_n)}{(b_1+\dots+b_n)} \otimes S(u,w) & H^*(\operatorname{BMaps}_0(\Sigma(\hat{g}, n), U_2)) \\
\bigotimes\limits_{i=1}^n \Lambda(b_i) \otimes S(u_i, w_i) \ar[u]^f \ar[r]^{\cong} & \bigotimes\limits_{i=1}^n H^*(BL_0U_2) \ar[u]^{B\pi^*}, } \end{gather*}
where $\deg(b_i) = 3$, $\deg(u_i) = \deg(u) =2$, $\deg(w_i) =\deg(w)=4$, $A$ is an exterior algebra with Poincar\'e series $P_t(A) = (1+t)^{2\hat{g}}(1+t^3)^{2\hat{g}}$, and $f(b_i) = b_i$, $f(u_i) =u$, $f(w_i) = w$.
\end{Lemma}

\begin{proof}
This is proven in Lemma~4.4 of~\cite{B} for ${\mathbb Z}_2$-coef\/f\/icients, but the proof actually works over~${\mathbb Z}$ and hence holds for any coef\/f\/icient f\/ield.
\end{proof}

The Koszul--Tate complex for the homomorphism $B\pi^*$ is identif\/ied with the bigraded complex $(K^{*,*},\delta)$ where
\begin{gather}\label{ktc}
 K^{*,*} := \Lambda(b_1,\dots,b_n, x_1,\dots,x_{n-1}, y_1,\dots,y_{n-1}) \otimes S(u_1,w_1,\dots, u_n,w_n) \otimes \Gamma(z)
\otimes A,
\end{gather}
where $\Gamma(z)$ is the divided power algebra generated by $z$, with bi-degrees and dif\/ferential $\delta\colon K^{i,j} \! \rightarrow K^{i+1, j}$ described in the table below.
\begin{center}
\begin{tabular}{|c|c|c|}
	\hline
generator & bi-degree & $\delta$-derivative \\
\hline
$u_i$ & $(0,2)$ & $0$ \\
$b_i$ & $(0,3)$ & $0$ \\
$w_i$ & $(0,4)$ & $0$\\
$x_i$ & $(-1,2)$ & $u_{i}-u_{n}$ \\
$y_i$ & $(-1,4)$ & $w_{i} -w_{n}$\\
$z$ & $(-1,3)$ & $b_1 +\dots+ b_n$\\
\hline
\end{tabular}
\end{center}

Note in particular that $K^{*,*}$ is a free module over $\bigotimes_{i=1}^n \Lambda(b_i) \otimes S(u_i, w_i)$ and the cohomology of $(K^{*,*},\delta) $ is isomorphic to $H^*(\operatorname{BMaps}_0(\Sigma(\hat{g}, n), U_2))$.

\subsection{Real loop groups}\label{orlecbularc}

Let $S^1 = {\mathbb R}/2\pi{\mathbb Z}$. Denote by $LU_2$ the group of continuous maps from $S^1$ into the unitary group~$U_2$ and by~$L_0U_2$ the subgroup of~$LU_2$ consisting of maps homotopic to a~constant map. The real loop groups $LU_2^{{\tilde{\tau}}_i}$ appearing in diagram (\ref{aoeurclorce...}) arise as subgroups of~$L_0U_2$. There are three kinds.
\begin{itemize}\itemsep=0pt
\item[(a)] $LU_2^{{\tilde{\tau}}_a} = LO_2$ sitting inside $LU_2$ in the standard way.
\item[(b)] $LU_2^{{\tilde{\tau}}_b}$ the gauge group of a rank two M\"obius bundle over $S^1$.
\item[(c)] $LU_2^{{\tilde{\tau}}_c} = \{ g\colon S^1 \rightarrow U_2 | g(\theta) = \overline{g(\theta + \pi)}\}$ where the bar means entry-wise complex conjugation.
\end{itemize}

There is one real loop group of type (a) for each real component of $\Sigma^{\tau}$ over which $E^{{\tilde{\tau}}}$ is trivial, one of type~(b) for each real component for which~$E^{{\tilde{\tau}}}$ is nonorientable, and a positive number of type (c) if and only if $\Sigma \setminus \Sigma^{\tau}$ is connected.

\begin{Lemma}\label{realloopgroupscoh}
Over coefficient fields $k$ of characteristic $\neq 2$ we have an isomorphism
\begin{gather*} H^*(BLU_2^{{\tilde{\tau}}};k) = \begin{cases} \Lambda (q) \otimes S(p) & \text{if ${\tilde{\tau}} = {\tilde{\tau}}_a$ or ${\tilde{\tau}}_c$}, \\
k & \text{if ${\tilde{\tau}} = {\tilde{\tau}}_b$}, \end{cases}\end{gather*}
where $\deg(q) = 3$ and $\deg(p) = 4$.
In the prior case, the restriction map \begin{gather*}\xymatrix{ H^*(BL_0U_2) \ar[r] \ar[d]^{\cong} & H^*(BLU_2^{{\tilde{\tau}}}) \ar[d]^{\cong} \\
\Lambda(b) \otimes S(u,w) \ar[r] & \Lambda(q) \otimes S(p) } \end{gather*}
sends $b$ to $q$, $w$ to $p$, and $u$ to $0$.
\end{Lemma}

\begin{proof}
(a) We start with the case $LU_2^{{\tilde{\tau}}_a} = LO_2$ which is surely well known (the two-fold cover $BLSO_2 \rightarrow BLO_2$ reduces the problem to the case of $LSO_2 = LU_1$, which was considered in Lemma 4.4 of \cite{B} for characteristic 2, but the proof works over any f\/ield). There is a canonical homotopy equivalence
$BLO_2 \cong \operatorname{Maps}_0(S^1, BO_2)$ with the function space of maps homotopic to a constant map. Consider the evaluation map
\begin{gather*}\operatorname{ev}\colon \ S^1 \times \operatorname{Maps}_0\big(S^1, BO_2\big) \rightarrow BO_2.\end{gather*}
If $p_1\in H^4(BO_2;k)$ denotes the f\/irst Pontryagin class, then we have an isomorphism $ H^*(BLO_2)$ $\cong \Lambda(q) \otimes S(p)$ where $p = \int_{[pt]} \operatorname{ev}^*(p_1) \in H^4(BLO_2; k)$ and $q = \int_{[S^1]}\operatorname{ev}^*(p_1) \in H^3(BLO_2;k)$ where~$\int$ denotes slant product with respect to the homology classes $[pt], [S^1]$.

The two remaining cases can be realized up to isomorphism as twisted loop groups (see Baird~\cite{B2}). Let $I=[0,1]$ be the unit interval, let~$G$ be a compact Lie group, and let $\sigma \in \operatorname{Aut}(G)$ be an automorphism. Then the associated \emph{twisted loop group} is
\begin{gather*} L_{\sigma}G := \{ g\colon I \rightarrow G\,|\, g(0) = \sigma(g(1))\}. \end{gather*}

(b) We have $LU_2^{{\tilde{\tau}}_b} \cong L_{\sigma}O_2$ where $\sigma$ is an orientation reversing orthogonal change of basis. Note that $\sigma$ restricts to an automorphism of $SO_2$. It was proven in \cite[Proposition~7.6]{B2} that for coef\/f\/icient f\/ields of characteristic other than two, \begin{gather*}H^*(BL_{\sigma}SO_2) \cong H^*(BLSO_1) \cong H^*(B{\mathbb Z}_2) \cong k.\end{gather*}
 Since $L_{\sigma}SO_2 \subset L_{\sigma} O_2$ is a subgroup of index two, we deduce that $H^*(BL_{\sigma}O_2) \cong H^*(BL_{\sigma}SO_2)^{{\mathbb Z}_2}$ is the subring of $H^*(BL_{\sigma}SO_2)$ invariant under the residual action by ${\mathbb Z}_2 \cong L_{\sigma}O_2 /L_{\sigma}SO_2$. In particular, $H^*(BL_{\sigma}O_2) \cong k^{{\mathbb Z}_2} \cong k$.

(c) Consider the map $I \hookrightarrow S^1$ that embeds $I$ as a half circle. Then restriction determines an isomorphism
\begin{gather}\label{typec}
LU_2^{{\tilde{\tau}}_c} \cong L_{\sigma}U_2,
\end{gather}
where $\sigma$ is entry-wise complex conjugation. By \cite[Corollary~7.5]{B2}, the inclusions $LO_2 \subset L_{\sigma} U_2$ induces an isomorphism \begin{gather*} H^*( L_{\sigma} U_2) \cong H^*(BLO_2) \cong \Lambda (q) \otimes S(p).\end{gather*}
Because~(\ref{typec}) is induced by including $I$ as a half circle in $S^1$, we have a commutative diagram
\begin{gather*}\xymatrix{ H^*(BLO_2) &H^*(BLO_2) \ar[l]^{f^*} \\
H^*(BLU^{{\tilde{\tau}}_c}_2) \ar[u]^{\cong} & H^*(BLU_2) \ar[l]^{i^*} \ar[u], }\end{gather*}
where the vertical arrows are induced by inclusion and $f^*$ is induced by a 2-fold covering map $f\colon S^1 \rightarrow S^1$.
Functoriality properties of the slant product imply that
\begin{gather*} f^*(p) = \int_{f_*[pt]} ev^*(p_1)) = \int_{[pt]} ev^*(p_1)) = p , \\
 f^*(q) = \int_{f_*[S^1]} ev^*(p_1)) = \int_{2[S^1]} ev^*(p_1)) = 2 q .\end{gather*}
Since $2$ is invertible, we can simply relabel~$2q$ as~$q$ as an element of $H^*(BLU_2^{{\tilde{\tau}}_c})$ completing the proof.
\end{proof}

\subsection{The spectral sequence}

We refer the reader to \cite[Appendix~A]{B} or McLeary \cite[Section~7.1]{M} for background on the Eilenberg--Moore spectral sequence.

\begin{proof}[Proof of Theorem~\ref{bgring}]
Denote $B{\mathcal G}^{{\tilde{\tau}}} = B{\mathcal G}(\hat{g}, n; {\tilde{\tau}}_1,\dots,{\tilde{\tau}}_n)$ from the homotopy pull-back diag\-ram~(\ref{aoeurclorce...}). The associated Eilenberg--Moore spectral sequence $EM_r^{*,*}$ converges to $H^*(B{\mathcal G}^{{\tilde{\tau}}})$. The second page $EM_2^{*,*}$, equals the cohomology of the bi-graded dif\/ferential graded algebra $ (K^{*,*} \otimes_{R^*} M^*, \delta \otimes 1) $
where
\begin{itemize}\itemsep=0pt
	\item $(K^{*,*}, \delta)$ is the Koszul--Tate complex~(\ref{ktc}),
	\item $M^* = M^{0,*} := \bigotimes_{i=1}^n H^*(BLU_2^{{\tilde{\tau}}_i})$, and
	\item $R^* = R^{0,*}:= \bigotimes_{i=1}^n H^*(BL_0U_2) = \bigotimes_{i=1}^n \Lambda(b_i) \otimes S(u_i,w_i). $
\end{itemize}
Suppose that there are $a \leq n$ real circles above which the real bundle is nonorientable and set $b = n-a$. Then by Lemma \ref{realloopgroupscoh} we have
\begin{gather*} \bigotimes_{i=1}^n H^*(BLU_2^{{\tilde{\tau}}_i}) \cong \otimes_{i=1}^{b} \Lambda(q_i) \otimes S(p_i)\end{gather*}
and
\begin{gather*} K^{*,*} \otimes_{R^*} M^* \cong \Lambda(q_1,\dots,q_{b}, x_1,\dots,x_{n-1}, y_1,\dots,y_{n-1}) \otimes S(p_1,\dots,p_{b}) \otimes \Gamma(z) \otimes A,
\end{gather*}
which comes with dif\/ferential $\delta' = \delta \otimes 1$ satisfying $\delta'(q_i) = \delta'(p_i) = \delta'(x_i) = \delta'(A) = 0$, and
\begin{gather*}\delta'(y_i) = \begin{cases}
 p_i -p_n & \text{if $b=n$}, \\
 p_i & \text{if $b<n$ and $i\leq b$},\\
 0 & \text{if $b<n$ and $i > b$},
 \end{cases}\\
\delta(z) = \begin{cases} q_1 +\dots+q_{b} & \text{if $b >0$},
\\ 0 & \text{ if $b =0$}. \end{cases}
\end{gather*}

The cohomology $EM_2^{*,*} =H(K^{*,*} \otimes_{R^*} M^*, \delta')$ is easily calculated in all cases and can be expressed in the original variables by abuse of notation.

In case $0<b<n$, we have \begin{gather*}EM_2^{*,*} \cong \frac{\Lambda(q_1,\dots,q_b)}{(q_1+\dots+q_b)} \otimes \Lambda(x_1,\dots,x_{n-1},y_{b+1},\dots,y_{n-1}) \otimes A.\end{gather*}

In case $b=n$ (i.e., the real bundle restricts to a trivial bundle over all real circles), we have
\begin{gather*} EM_2^{*,*} \cong \frac{\Lambda(q_1,\dots,q_n)}{(q_1+\dots+q_n)} \otimes \Lambda(x_1,\dots,x_{n-1})\otimes S(p) \otimes A, \end{gather*}
where $p$ is represented by cocycle $p_1$.
In case $b=0$ (i.e., the real circle are separating and the real bundle restricts to nonorientable bundles for all of them), then $\delta' =0$ and we have
 \begin{gather*}EM_2^{*,*} \cong \Lambda(x_1,\dots,x_{n-1},y_{1},\dots,y_{n-1}) \otimes \Gamma(z) \otimes A.\end{gather*}

If $b>0$, then the bigraded algebra $EM_2^{*,*}$ is generated by homogenous elements lying in $EM_2^{0,q}$ or $EM_2^{-1,q}$ for some $q$ (i.e., the $0$th and $-1$th columns). For degree reasons the generators must survive until inf\/inity, so the spectral sequence must collapse. Since $EM_2^{*,*}$ is a free graded-commutative algebra and an associated graded algebra of~$H^*(B{\mathcal G}^{\tilde{\tau}})$, we deduce that $H^*(B{\mathcal G}^{\tilde{\tau}})$ is a free super commutative algebra isomorphic to~$EM_2^{*,*}$.

If $b=0$, then the preceding argument still works for coef\/f\/icient f\/ields of characteristic zero f\/ields, because in that case $\Gamma(z) = S(z)$. The universal coef\/f\/icient theorem then implies that the spectral sequence collapses for odd characteristic f\/ields as well.
\end{proof}

\section{Proofs of Theorem \ref{bigthm} and Corollary \ref{BigCor}}

\begin{proof}[Proof of Theorem \ref{bigthm}]
For the generic case, simply combine Theorems~\ref{isomcoh} and~\ref{bgring}, and Proposition~\ref{orientableProp}.

It remains to show that $H^*(M(E,\tilde{\tau});k)$ is not an exterior algebra if $(\Sigma,\tau)$ is type~I of even genus greater than two, $E$ has odd degree, and $E^{\tilde{\tau}}$ is nonorientable on all components of $\Sigma^{\tau}$. According to Theorem~\ref{bgring}, the cohomology ring $H^*(B{\mathcal G}^{\tilde{\tau}})$ has Poincar\'e series
\begin{gather}\label{quickref}
P_t\big(B{\mathcal G}^{\tilde{\tau}}\big) = (1+t)^g\big(1+t^3\big)^g/\big(1-t^2\big).
\end{gather} According to Proposition \ref{orientableProp}, the normal bundles of all strata are orientable, so the Thom spaces satisfy the Thom isomorphism. Since the unstable strata all have codimension greater than $(g-1)$ (see Section~\ref{realstratsect}), then consideration of Thom--Gysin sequences yields equalities for Betti numbers:
\begin{gather*}\dim(H^i(M(E,\tilde{\tau}))) = \dim(H^i( B{\mathcal G}^{\tilde{\tau}})),\qquad \text{for $i\leq g-2$}.\end{gather*}
In particular, if $g\geq 4$ then the f\/irst and second Betti numbers of $M(E,\tilde{\tau})$ and $B{\mathcal G}^{\tilde{\tau}}$ must agree. From~(\ref{quickref}), the f\/irst Betti number of $B{\mathcal G}^{\tilde{\tau}}$ is $g$ and the second is ${ g \choose 2}+1$ which implies that $H^*(M(E,\tilde{\tau});k)$ cannot possibly be an exterior algebra.
\end{proof}

Lastly, we explain the relationship between the cohomology of $M(r,d,\tilde{\tau})$ and that of the f\/ixed determinant moduli space $M_{\Lambda}(r,d,\tilde{\tau})$. This is analogous to the complex version described in Atiyah--Bott~\cite[p.~578]{AB}.

Consider the trivial $C^{\infty}$-line bundle $\Sigma \times {\mathbb C}$, with trivial real structure \begin{gather*}\bar{c}(x,z) = (\tau(x),\bar{z}).\end{gather*}
The moduli space $M(1,0,\bar{c})$ is a group isomorphic to $(S^1)^g$. The subgroup $T_r \leq M(1,0,\bar{c})$ of $r$-th roots of unity acts by tensor product on $M_{\Lambda}(r,d,\tilde{\tau})$. The tensor product map $ M_{\Lambda}(r,d,\tilde{\tau}) \times M(1,0, \bar{c}) \rightarrow M(r,d,\tilde{\tau})$ is a covering space with transitive deck transformation group $T_r$, so it descends to a dif\/feomorphism
\begin{gather*} M(r,d,\tilde{\tau}) \cong M_{\Lambda}(r,d,\tilde{\tau}) \times_{T_r} M(1,0, \bar{c}),\end{gather*}
where we have taken the mixed quotient with respect to the tensor product action of $T_r$ on $M_{\Lambda}(r,d,\tilde{\tau})$ and $M(1,0, \bar{c})$. Since $T_r$ acts by translations on $M(1,0, \bar{c})$, it acts trivially on cohomology. It follows that for characteristic relatively prime to $r$, we have a ring isomorphism
	\begin{gather*} H^*(M(r,d,\tilde{\tau});k) \cong H^*(M_{\Lambda}(r,d,\tilde{\tau});k)^{T_r} \otimes H^*(M(1,0, \bar{c});k ). \end{gather*}
	Finally, $M(1,0, \bar{c})$ is dif\/feomorphic to $(S^1)^g$, so $H^*(M(1,0, \bar{c});k )$ is an exterior algebra with $g$ generators in degree one. Corollary \ref{BigCor} follows.

\subsection*{Acknowledgements}
A special thanks to Jacques Hurtubise and Ben Smith who began this project as collaborators and contributed to some of the exposition. Jacques in particular helped motivate this project by establishing criteria for the normal bundles of the real Harder--Narasimhan stratif\/ication to be non-orientable (a proof later superseded by the work of Okonek--Teleman). Thanks also to Andrei Teleman and other the participants at the Real vector bundles conference in Brest for helpful discussions, and to the referees for helpful comments. This research was supported by an NSERC Discovery Grant.

\pdfbookmark[1]{References}{ref}
\LastPageEnding

\end{document}